\documentclass[12pt]{amsart}
\usepackage{fullpage,url,amssymb,enumerate,colonequals}
\usepackage{amsmath}
\usepackage{amsfonts}
\usepackage{amssymb}
\usepackage{mathrsfs}
\usepackage{enumerate}
\usepackage{times}
\usepackage{enumitem}%\usepackage{showkeys}  
\numberwithin{equation}{section}

\usepackage{hyperref}
\hypersetup{
    colorlinks=true,
    linkcolor=blue,
    filecolor=magenta,
    citecolor = red,
    urlcolor=red,
    pdftitle={Overleaf Example},
    pdfpagemode=FullScreen,
    }

\newtheoremstyle{myplain}
                {}          % Space above
                {}          % Space below
                {\itshape}  % Body font
                {20pt}          % Indent amount
                {\scshape} % Head font
                {.}         % Punctuation after head
                { }         % Space after theorem head
                {}          % Theorem head spec
\theoremstyle{myplain}

\newtheoremstyle{mydefinition}
                {}          % Space above
                {}          % Space below
                {}  % Body font
                {20pt}          % Indent amount
                {\scshape} % Head font
                {.}         % Punctuation after head
                { }         % Space after theorem head
                {}          % Theorem head spec
\theoremstyle{mydefinition}

\theoremstyle{plain}

\newtheorem{theorem}{Theorem}[section]

\newtheorem{lemma}[theorem]{Lemma}
\newtheorem{corollary}[theorem]{Corollary}

\theoremstyle{definition}

\renewcommand{\leq}{\leqslant}
\renewcommand{\geq}{\geqslant}

\newsavebox{\proofbox}
\savebox{\proofbox}{\begin{picture}(7,7)%
  \put(0,0){\framebox(7,7){}}\end{picture}}

\def\Q{{\mathbb Q}}

\def\emph#1{{\it #1}}
\def\textbf#1{{\bf #1}}

\title{Some Series related to Extended Riemann Hypothesis for Dedekind zeta functions}
\author{Muhammad Atif Zaheer}
\address{Department of Mathematics\\Pennsylvania State University, University Park\\ State College PA 16802}
\email{mvz5421@psu.edu}

\date{}



\begin{document}

\begin{abstract}
We obtain closed form of some infinite series involving derivatives of an analogue of the Riemann xi function for Dedekind zeta function and nontrivial zeros of Dedekind zeta function assuming the Extended Riemann Hypothesis. Conversely, we prove that if this closed form holds, then all of the zeros of Dedekind zeta function beyond a certain height lie on the critical line. This yields a large number of equivalent statements of Riemann Hypothesis.

\end{abstract}

\maketitle

\section{Introduction}

The Dedekind zeta function, $\zeta_K(s)$, associated to a number field $K$ of degree $n$ over $\Q$, is defined for $\sigma > 1$ as the absolutely convergent series
\[
\zeta_K(s) = \sum_{\mathfrak{a}} \frac{1}{N_{K/\Q}(\mathfrak{a})^s},
\]
where the sum runs through all nonzero ideals $\mathfrak{a}$ of the ring of integers $\mathcal{O}_K$ of $K$ and $N(\mathfrak{a})$ denotes the norm of a nonzero ideal $\mathfrak{a}$ in $\mathcal{O}_K$, i.e., $N(\mathfrak{a}) = [\mathcal{O}_K : \mathfrak{a}]$. For the special case $K = \mathbb{Q}$,  the Dedekind zeta function reduces to the classical Riemann zeta function; $\zeta_{\mathbb{Q}}(s) = \zeta(s)$.

Analogous to the Riemann zeta function, Dedekind zeta functions also have Euler product representation
\begin{equation} \label{euler product}
\zeta_K(s) = \prod_{\mathfrak{p}}  \frac{1}{1 - N_{K/\Q}(\mathfrak{p})^{-s}}
\end{equation}
for $\sigma >1$, where the product is extended over all prime ideals $\mathfrak{p}$ of $\mathcal{O}_K$. This is the result of the uniqueness of prime factorization of nonzero ideals in $\mathcal{O}_K$. As a consequence, $\zeta_K(s)$ does not vanish in the half plane $\sigma > 1$.

As in the case of Riemann zeta function Dedekind zeta function also admits a unique analytic continuation to the whole complex plane except for a simple pole at $s = 1$. Dedekind proved the analytic class number formula, showing that $\zeta_K(s)$ has a simple pole at $s = 1$ with residue
\begin{equation} \label{class number formula}
\frac{2^{r_1 + r_2} \pi^{r_2} R_K h_K}{\omega_K |\Delta_K|^{1/2}},
\end{equation}
where $r_1$ and $r_2$ are the number of real embeddings and pairs of complex embeddings respectively, $R_K$ is the regulator of $K$, $h_K$ is the class number of $K$, $\omega_K$ is the number of roots of unity contained in $K$, and $\Delta_K$ is the discriminant of the field extension $K/\Q$. This is an instance of the ``local-global principle"; the Dedekind zeta function $\zeta_K(s)$ is assembled from local components, namely the prime ideals of $\mathcal{O}_K$, as expressed by the Euler product \eqref{euler product} yet it also reflects global invariants of the field in an integrated fashion, as is revealed by the class number formula \eqref{class number formula}.

\vspace{2pt}

Dedekind zeta function satisfies the functional equation
\[
\zeta_K(1-s) = |\Delta_K|^{s - 1/2} 2^{n(1-s)} \pi^{-ns} \left( \cos \frac{\pi s}{2} \right)^{r_1 + r_2}  \left( \sin \frac{\pi s}{2} \right)^{r_2} \Gamma(s)^n \zeta_K(s),
\]
which was discovered by Hecke \cite{hecke1917zetafunktion}. A detailed exposition of Hecke's proof of the functional equation for $\zeta_K(s)$ can be found in \cite[Chap.~7, §5]{Neukirch}. It reduces to the functional equation of $\zeta(s)$ for the case $K = \Q$. Furthermore, it can be seen from the above functional equation that $\zeta_K(s)$ has a zero of order $r_1 + r_2 - 1$ at $s = 0$ since the term involving cosine contributes a zero of order $r_1 + r_2$ at $s = 1$, $\zeta_K(s)$ has a simple pole at $s = 1$ and other terms are analytic and do not vanish at $s = 1$.

\section{Some preliminaries}

One of the most notoriously difficult problem in all of mathematics is the Riemann Hypothesis (RH) which asserts that all of the nontrivial trivial zeros $\rho = \beta + i \gamma$ of $\zeta(s)$ satisfy $\beta = 1/2$. One can conjecture a similar statement for Dedekind zeta functions: All of the zeros $\rho_K = \beta_K + i \gamma_K$ of $\zeta_K(s)$ satisfying $0 < \beta_K < 1$ lie on the critical line, i.e., $\beta_K = 1/2$. This is usually referred as the Extended Riemann Hypothesis (ERH). 

We have a de la Vall\'{e}e Poussin type zero-free region for $\zeta_K(s)$. If $n$ is the degree of field extension $K/\Q$ and $d$ is the discriminant, then there exists an absolute constant $c > 0$ such that $\zeta_K(s)$ has no zeros in the region 
\[
\sigma > 1 - \frac{c}{n^2 \log  (|d|(|t| + 2)^{n}) }
\]
\cite[Thm. 5.33]{iwaniec2021analytic}. There could potentially be a real zero very close to the line $\sigma  =1$ which is referred as Landau-Siegel zero. It is known that there could be at most one such zero and if it exists, then it must be simple. The existence or nonexistence of this zero has very important consequences. It was shown by Heath-Brown \cite{heath1983prime} that the nonexistence of Landau-Siegel zeros for Dirichlet L-functions implies twin prime conjecture.

We consider an analogue of the Riemann xi function, or also known as the completed Riemann zeta function, for Dedekind zeta function $\zeta_K(s)$ as
\begin{equation} \label{definition of Dedekind xi function}
\xi_K(s) = \frac{1}{2}s(s-1) |\Delta_K|^{s/2} 2^{(1-s)r_2} \pi^{-ns/2} \Gamma \left( \frac{s}{2} \right)^{r_1} \Gamma(s)^{r_2} \zeta_K(s).
\end{equation}
It is an entire function of order $1$ and satisfies the functional equation 
\begin{equation} \label{functional equation of completed Dedekind xi-function}
\xi_K(s) = \xi_K(1-s).
\end{equation}
$\xi_K(s)$ reduces to $\xi(s)$ for the case $K = \Q$. Moreover, $\xi_K(s)$ does not vanish at $s = 0$ and the zeros of $\xi_K(s)$ coincide with the zeros of $\zeta_K(s)$ in the critical strip $0 < \sigma < 1$.

\begin{lemma} \label{vanishing of xi_K at 1/2}
For $k$ odd, $\xi_K^{(k)}(1/2) = 0$.
\end{lemma}

\begin{proof}
From the functional equation \eqref{functional equation of completed Dedekind xi-function} we readily obtain functional equation for $\xi_K^{(k)}(s)$ by repeated differentiation as
\[
\xi_K^{(k)}(s) = (-1)^{k} \xi_K^{(k)}(1-s).
\]
The result is now immediate by plugging $s = 1/2$.
\end{proof}

\begin{lemma}
If $m$ is the order of vanishing of $\zeta_K(s)$ at $s = 1/2$, then $m$ is even.
\end{lemma}

\begin{proof}
If $\zeta_K(1/2) \neq 0$, then $m = 0$ and so the conclusion is trivial. From the definition \eqref{definition of Dedekind xi function} of $\xi_K(s)$ we see that $m$ is same as the order of vanishing of $\xi_K(s)$ at $s = 1/2$.  Now suppose that $\zeta_K(1/2) = 0$. Since $\xi_K(s)$ is entire, it admits power series expansion 
\[
\xi_K(s) = \sum_{k = 0}^{\infty} \frac{\xi_K^{(k)}(1/2)}{k!} \left(s - \frac{1}{2} \right)^{k}.
\]
Because $\xi_K(1/2) = 0$ and $m$ is the smallest power $k$ for which $\xi_K^{(k)}(1/2) \neq 0$, it follows by Lemma \ref{vanishing of xi_K at 1/2} that $m$ must be even.
\end{proof}

We now define
\[
\Xi_K(s) = \left(s - \frac{1}{2} \right)^{-m} \xi_K(s),
\]
where $m$ is the order of vanishing of $\zeta_K(s)$ at $s = 1/2$. Since $m$ is even, $\Xi_K(s)$ satisfies the functional equation 
\[
\Xi_K(s) = \Xi_K(1-s).
\]

\begin{lemma}
For $k$ odd, $\Xi_K^{(k)}(1/2) = 0$.
\end{lemma}

\begin{lemma} \label{vanishing of X_K at 1/2} 
For $k$ even, $(\Xi_K'/\Xi_K)^{(k)}(1/2) = 0$.
\end{lemma}

Because $\Xi_K(s)$ is complex-conjugate symmetric, i.e., $\overline{\Xi_K(s)} = \Xi_K(\overline{s})$, so are its higher derivatives and as a result $\Xi_K^{(k)}(s)$ is real-valued on the real-axis. For a detailed proof see \cite[Lem. 3]{nguyen2025some}.

Since $\Xi_K(s)$ is an entire function of order $1$, it can be expressed as a Hadamard product
\begin{equation} \label{hadamard product}
\Xi_K(s) = e^{A + Bs} \prod_{\rho_K \neq 1/2} \left( 1 - \frac{s}{\rho_K} \right) e^{s/\rho_K},
\end{equation}
where $A$ and $B$ are constants and the product runs over the zeros of $\zeta_K(s)$ that are distinct from $1/2$ and lie in the critical strip $0 < \sigma < 1$. This product converges uniformly on compact sets not containing any zeros $\rho_K \neq 1/2$.

\vspace{2pt}

\section{Main result}

A plethora of equivalent statements of RH have been discovered in the last century. A detailed account of both arithmetic and analytic equivalents can be found in \cite{Broughan1} and \cite{Broughan2} respectively. Recently, Suman and Das \cite{Suman} showed that RH is equivalent to 
\[
\sum_{\rho} \frac{1}{|1/2 - \rho|^2} = \frac{\xi''(1/2)}{\xi(1/2)}.
\]
Later in \cite{suman2022note}, they showed that RH holds if and only if
\[
\sum_{\rho} \frac{1}{|1/2 - \rho|^4} = \frac{1}{2} \frac{\xi''(1/2)^2}{\xi (1/2)^2} - \frac{1}{6} \frac{\xi^{(4)}(1/2)}{\xi(1/2)}.
\]
Very recently, Nguyen \cite{nguyen2025some} extended Suman and Das result to Dedekind zeta functions and showed that ERH for $\zeta_K(s)$ is equivalent to 
\[
\sum_{\rho_K \neq 1/2} \frac{1}{|1/2 - \rho_K|^2} = \frac{\Xi_K''(1/2)}{\Xi_K(1/2)}.
\]
We generalize Suman and Das result and obtain a large number of equivalent statements for RH involving closed form of sums involving $|1/2 - \rho|^{k + 1}$, where $k$ is an odd positive integer. We do so in a general setting of Dedekind zeta functions, as is done by Nguyen for the special case $k = 1$.

\begin{theorem}
Let $K$ be a number field and let $k > 1$ be an odd positive integer. If ERH holds, then 
\begin{equation} \label{main result}
\frac{(-1)^{(k - 1)/2}}{k!} \left( \frac{\Xi_K'}{\Xi_K} \right)^{(k)} \left(\frac{1}{2} \right) = \sum_{\rho_K \neq 1/2} \frac{1}{|1/2 - \rho_K|^{k + 1}},
\end{equation}
where the sum runs over all nontrivial zeros $\rho_K$ of $\zeta_K(s)$ distinct from $1/2$. Moreover, if the above equality holds, then $\beta_K = 1/2$ for $|\gamma_K| > (1/2)\cot (2 \pi/(k + 1))$.
\end{theorem}

\begin{proof}
Taking the logarithmic derivative of the Hadamard product \eqref{hadamard product} we obtain
\[
\frac{\Xi_K'}{\Xi_K}(s) = B + \sum_{\rho_K \neq 1/2} \left( \frac{1}{s - \rho_K} + \frac{1}{\rho_K} \right).
\]
Because the sum converges uniformly on compact subsets not containing the nontrivial zeros of $\zeta_K(s)$ we can differentiate term by term and deduce that
\[
\left( \frac{\Xi_K'}{\Xi_K} \right)'(s) = - \sum_{\rho_K \neq 1/2} \frac{1}{(s - \rho_K)^2}.
\]
By repeated differentiation we end up with the identity
\[
\frac{(-1)^k}{k!}\left( \frac{\Xi_K'}{\Xi_K} \right)^{(k)}(s) = \sum_{\rho_K \neq 1/2} \frac{1}{(s - \rho_K)^{k + 1}}
\]
which holds for any positive integer $k$. For $k$ even it immediately follows from Lemma \ref{vanishing of X_K at 1/2} that
\[
\sum_{\rho_K \neq 1/2} \frac{1}{(1/2 - \rho_K)^{k + 1}} = 0.
\]
But this is well-known and follows simply due to the symmetry of zeros $\rho_K$ under conjugation and reflection across the critical line. We now assume that $k$ is odd. By noting that $(\Xi_K'/\Xi_K)^{(k)} (1/2)$ is real we obtain
\begin{align*}
\frac{(-1)^k}{k!}\left( \frac{\Xi_K'}{\Xi_K} \right)^{(k)} \left( \frac{1}{2} \right) &= \frac{1}{2}\sum_{\rho_K \neq 1/2} \frac{1}{(1/2 - \rho_K)^{k + 1}} + \frac{1}{2} \sum_{\rho_K \neq 1/2} \frac{1}{(1/2 - \overline{\rho_K})^{k + 1}}  \\
&= \frac{1}{2} \sum_{\rho_K \neq 1/2} \frac{(1/2 - \rho_K)^{k + 1} + (1/2 - \overline{\rho_K})^{k + 1}}{|1/2 - \rho_K|^{2(k + 1)}} \\
&= \sum_{\rho_K \neq 1/2} \frac{\Re(1/2 - \rho_K)^{k + 1}}{|1/2 - \rho_K|^{2(k + 1)}}.
\end{align*}
If ERH holds, then $\Re(1/2 - \rho_K)^{k + 1} = (-1)^{(k + 1)/2} |1/2 - \rho_K|^{k + 1}$ and so we get \eqref{main result}.

Now suppose that \eqref{main result} holds. Then we have 
\begin{equation} \label{second main equation}
\sum_{\rho_K \neq 1/2} \frac{\Re(1/2 - \rho_K)^{k + 1} - (-1)^{(k + 1)/2}|1/2 - \rho_K|^{k + 1}}{|1/2 - \rho_K|^{2(k + 1)}} = 0.
\end{equation}
If $k \equiv 1 \pmod{4}$, then the numerator in the summand is always nonnegative and so we have $\Re(1/2 - \rho_K)^{k + 1} = - |1/2 - \rho_K|^{k + 1}$ for every zero $\rho_K \neq 1/2$. Let $1/2 - \rho_K = |1/2 - \rho_K| e({ \phi_{\rho_K}})$, where we use $e(x)$ to denote $e^{ix}$ and $\phi_{\rho_K}$ is the principal argument of $1/2 - \rho_K$. Then it follows that
\[
\Re \left( \left|1/2 - \rho_K \right|^{k + 1} e({(k + 1) \phi_{\rho_K}}) \right)= - \left|1/2 - \rho_K \right|^{k + 1}
\]
for every zero $\rho_K \neq 1/2$, which in turn implies that $\cos((k + 1) \phi_{\rho_K}) = -1$ and thus $\phi_{\rho_K} =((2n + 1) \pi)/(k + 1)$, where $0 \leq n \leq  k$. Because $-1/2 < 1/2 - \beta_K < 1/2$, it follows that if $|\gamma_K|> (1/2) \cot (2 \pi/(k + 1))$, then $\tan^{-1}(2|\gamma_K|) > \pi/2 - 2\pi/(k + 1)$. This yields $|\phi_{\rho_K} - \pi/2| < 2 \pi/(k + 1)$  or $|\phi_{\rho_K} - 3 \pi/2| < 2 \pi/(k + 1)$. Consequently we must have either $\phi_{\rho_K} = \pi/2$ or $3 \pi/2$  and hence $1/2- \beta_K =  \Re(1/2  -\rho_K) = |1/2 - \rho_K|\cos(\phi_{\rho_K}) = 0$, i.e.,  $\beta_K = 1/2$.

If $k \equiv 3 \pmod{4}$. Then numerator in each summand in the sum \eqref{second main equation} is nonpositive for every zero $\rho_K \neq 1/2$. Thus $\Re (1/2 - \rho_K)^{k + 1} = |1/2 - \rho_K|^{k + 1}$ and so
\[
\Re \left(\left| 1/2 - \rho_K \right|^{k + 1} e({(k + 1) \phi_{\rho_K}}) \right)=  \left| 1/2 - \rho_K \right|^{k + 1}
\]
for every zero $\rho_K \neq 1/2$, which leads to $\cos((k + 1) \phi_{\rho_K}) = 1$ and so $\phi_{\rho_K} = 2\pi n/(k + 1)$, where $0 \leq n \leq k$. Again it follows that if $|\gamma_K| > (1/2) \cot (2 \pi/(k + 1))$, then we must have $\phi_{\rho_K} =  \pi/2$ or $3 \pi/2$. Hence $1/2 - \beta_K = \Re(1/2 - \rho_K) = |1/2 - \rho_K| \cos (\phi_{\rho_K}) = 0$, i.e., $\beta_K = 1/2$. This completes the proof.
\end{proof}

Observe that the height beyond which all zeros $\rho_K$ of $\zeta_K(s)$ lie on the critical line is uniform in $K$ and depends only on the order $k$ of the derivative of $\Xi_K'/\Xi_K$.

\begin{corollary}
The statement \eqref{main result} is equivalent to ERH for $k = 3$.
\end{corollary}

\begin{proof}
Follows simply by observing that $\cot (2 \pi/(k + 1)) = 0$ for $k = 3$.
\end{proof}

\begin{corollary} \label{main corollary}
Let $k$ be an odd positive integer. If RH holds, then 
\begin{equation} \label{equivalent of RH}
\frac{(-1)^{(k - 1)/2}}{k!} \left( \frac{\xi'}{\xi} \right)^{(k)} \left( \frac{1}{2} \right) = \sum_{\rho} \frac{1}{|1/2 - \rho|^{k + 1}},
\end{equation}
where the sum is taken over all nontrivial zeros $\rho$ of $\zeta(s)$. Moreover, if the above equality holds, then $\beta = 1/2$ for $|\gamma| > (1/2) \cot(2 \pi/(k + 1))$
\end{corollary}

\begin{corollary}
If $k <  12 \pi \cdot 10^{12}$, then \eqref{equivalent of RH} is equivalent to RH.
\end{corollary}

\begin{proof}
Due to inequality $\cot^{-1} (x) \leq 1/x$ for $x > 0$ we have
\[
\frac{2 \pi}{\cot^{-1}(6 \cdot 10^{12})} \geq 12 \pi \cdot 10^{12} \geq k + 1
\]
which in turn implies that
\[
\frac{1}{2} \cot \left(\frac{2 \pi}{k + 1} \right) \leq 3 \cdot 10^{12}.
\]
Since $\beta = 1/2$ for $|\gamma| \leq 3 \cdot 10^{12}$ \cite{platt2021riemann}, it follows by Corollary \ref{main corollary} above that RH holds.
\end{proof}

\bibliographystyle{alpha}
\bibliography{main.bib}

\end{document}